\documentclass[11pt]{article}
\usepackage{amsmath,amssymb,enumerate} 

\usepackage{amsthm}

\usepackage{tikz}
\usetikzlibrary{arrows}
\usetikzlibrary{positioning}

\theoremstyle{plain}
\newtheorem{theorem}{Theorem}

\newtheorem{proposition}{Proposition}
\newtheorem{corollary}{Corollary}
\newtheorem{conjecture}{Conjecture}

\theoremstyle{definition}
\newtheorem{remark}{Remark}

\def\cP{\mathcal{P}}

\def\D{{\cal D}}

\def\ZZ{\mathbb{Z}}

\def\NN{\mathbb{N}}

\DeclareMathOperator{\supp}{supp}

\title{More about Exact Slow $k$-{\sc Nim}}

\author{
Nikolay Chikin\thanks{
National Research University Higher School of Economics (HSE), Moscow, Russia; e-mail:
nikolajchikin@yandex.ru}
\and
Vladimir Gurvich\thanks{
National Research University Higher School of Economics (HSE), Moscow, Russia; e-mail:
vgurvich@hse.ru and vladimir.gurvich@gmail.com}
\and
Konstantin Knop\thanks{HIL company (Horis International Limited);
e-mail:
kostyaknop@gmail.com}
\and
Mike Paterson\thanks{
Department of Computer Science, University of Warwick, UK; e-mail:
M.S.Paterson@warwick.ac.uk}
\and
Michael Vyalyi\thanks
{National Research University Higher School of Economics, Moscow, Russia;
  Institute of Physics and Technology, Dolgorpudnyi, Russia; Dorodnicyn Computing Centre,
  FRC CSC RAS, Moscow, Russia;
 e-mail:
vyalyi@gmail.com}
}

\begin{document}

\maketitle

\centerline{\bf Abstract}
\noindent
Given  $n$  piles of tokens and a positive integer $k \leq n$,
the game  {\sc Nim}$^1_{n, =k}$  of exact slow $k$-{\sc Nim}  is played as follows.
Two players move alternately.
In each move, a player chooses exactly $k$ non-empty piles
and removes one token from each of them.
A player whose turn it is to move but has no move loses
(if the normal version of the game is played, and wins if it is the mis\'{e}re version).
In Integers 20 (2020) 1--19, Gurvich et al gave an explicit formula
for the Sprague-Grundy function of  {\sc Nim}$^1_{4, =2}$,
for both its normal and mis\'{e}re version.
Here we extend this result and
obtain an explicit formula for the P-positions of the normal version
of {\sc Nim}$^1_{5, =2}$  and  {\sc Nim}$^1_{6, =2}$.

\smallskip

{\bf Key words:}
Exact {\sc Nim}, normal and mis\'{e}re versions,
P-positions,\\ Sprague-Grundy function.
\newline
MSC classes: 91A46, 91A05

\section{Introduction and main results}
\label{s0}

Games {\sc Nim}$^1_{n, =k}$  and  {\sc Nim}$^1_{n, \leq k}$
of Exact and Moore's Slow $k$-{\sc Nim}  were introduced in 2015 \cite{GH15}.
The present paper extends some results obtained in  \cite{GHHC20}
for  {\sc Nim}$^1_{4, =2}$  to {\sc Nim}$^1_{5, =2}$   and  {\sc Nim}$^1_{6, =2}$.
All basic definitions (impartial games in normal and  mis\'{e}re versions,
positions, moves, P- and N-positions, Sprague-Grundy  function)
can be found in the introduction of  \cite{GHHC20}; so we will not repeat them.
Here we need only
P-positions of the normal version of {\sc Nim}$^1_{n, =k}$.

\medskip

Positions of  {\sc Nim}$^1_{n, =k}$  are represented by
nonnegative integer $n$-vectors $x = (x_1, \ldots, x_n)$.
An $n$-vector  $x$  is called {\em nondecreasing} if
$x_1 \leq  \ldots  \leq x_n$.
We will always  assume that
positions of  {\sc Nim}$^1_{n, =k}$  are represented
by nondecreasing vectors,
yet this assumption may hold for $x$  but fail for $x'$
after a move  $x \to x'$.
In this  case we reorder coordinates of  $x'$
to maintain the assumption; see more details in \cite{GHHC20}.

\medskip

Note that
{\sc Nim}$^1_{n', =k}$  is a subgame of
{\sc Nim}$^1_{n, =k}$  whenever  $n' \leq  n$.
Indeed, the set of  $n$-vectors  $x =   (x_1, \ldots, x_n)$
(which are positions of  {\sc Nim}$^1_{n, =k}$)
whose first  $n - n'$  coordinates are zeros
is in an obvious  one-to-one correspondence
with the set of  $n'$-vectors
(which are the positions of  {\sc Nim}$^1_{n', =k}$).
Thus, a formula for the P-positions or the Sprague-Grundy function
of {\sc Nim}$^1_{n, =k}$  works for {\sc Nim}$^1_{n', =k}$  as well.

\medskip

By definition, any move  $x \to x'$  in  {\sc Nim}$^1_{n, =k}$
reduces exactly  $k$  coordinates of  $x$  by exactly one  each;
hence,  $\sum_{i=1}^n x_i  = \sum_{i=1}^n x'_i  \mod k$.
Thus {\sc Nim}$^1_{n, =k}$  is split into
$k$  disjoint subgames  {\sc Nim}$^1_{n, =k}[j]$
for  $j = 0, \ldots, k-1$  such that  $\sum_{i=1}^n x_i = j \mod k$.

\bigskip

Explicit formulas were obtained in \cite{GHHC20}
for the Sprague-Grundy function of  both the normal and mis\'{e}re
versions of {\sc Nim}$^1_{4, =2}$.
Here we extend this result and
give explicit formulas for the P-positions of the normal versions
of {\sc Nim}$^1_{5, =2}$ and {\sc Nim}$^1_{6, =2}$.
We will assume that
{\sc Nim}$^1_{n, =k}$   and  {\sc Nim}$^1_{n, =k} [j]$
refer to the normal version of the game
unless it is explicitly  said otherwise.

\medskip

Given   $x = (x_1, \ldots, x_n)$,  its
{\em parity vector} $p(x) = (p(x_1, \ldots, p(x_n))$
is defined as follows: $p(x)$  has  $n$ coordinates
taking values  $p(x_i) = e$  if  $x_i$ is even and
$p(x_i) = o$  if  $x_i$ is odd,  for $i = 1,\ldots,n$.

\begin{theorem}
\label{t6}
The P-positions of   {\sc Nim}$^1_{6, =2}[0]$
are characterized by the parity vectors:
\begin{equation}
\label{eq-t6}
(e, e, o, o, o, o), (o, o, e, e, o, o), (o, o, o, o, e, e), (e, e, e, e, e, e).
\end{equation}
\end{theorem}

\begin{remark}
\label{r1}
Somewhat surprisingly, the same characterization holds
for the P-positions of Moore's Slow
{\sc Nim}$^1_{6, \leq 3}$; see part (5) of Theorem 2 in~\cite{GHHC20}.
\end{remark}

The set of P-positions of {\sc Nim}$^1_{6, =2}[0]$ can be defined by a
system of linear equations modulo 2 and the nondecreasing condition.
More specifically, the system of equations is
\begin{equation}\label{E*condition}
  \begin{aligned}
    x_2-x_1\equiv 0& \pmod2,\\
    x_4-x_3\equiv 0& \pmod2,\\
    x_5-x_4-x_1\equiv0& \pmod2.
  \end{aligned}
\end{equation}
Note that Eq.~\eqref{E*condition} implies $x_6-x_5\equiv0\pmod2$ since
\[
x_6-x_5+(x_4-x_3) +(x_2-x_1)\equiv x_6+x_5+x_4+x_3+x_2+x_1\equiv 0\pmod 2
\]
(the total number of tokens is even). It is easy to see that
Eq.~\eqref{eq-t6} provides all four solutions of
Eq.~\eqref{E*condition} such that  $x_6\equiv x_5\pmod2$.

The following concept will play an important role.
Given a  {\sc Nim}$^1_{n, =k}$  or {\sc Nim}$^1_{n, =k}[j]$,
a nonnegative nondecreasing  $n$-vector  $y$   is called a {\em P-shift}
if for any nondecreasing $n$-vector  $x$  we have:
either both  $x$  and  $x+y$  are  P-positions of the considered game,
or both are not.
Obviously, if  $y'$  and $y''$  are P-shifts then
$y = y' + y''$  is a P-shift too.

By Theorem \ref{t6}, the
nonnegative nondecreasing  $6$-vectors with even coordinates
and also vectors $(0,0,1,1,1,1)$ and $(1,1,1,1,2,2)$
are  P-shifts  in {\sc Nim}$^1_{6,
  =2}[0]$. (Eq.~\eqref{E*condition} helps to check this claim.)

As for {\sc Nim}$^1_{6, =2}[1]$, the set of P-positions has a more
complicated structure. To describe it, we introduce several
conditions on a nondecreasing 6-vector $(x_1,x_2,x_3,x_4,x_5,x_6)$.

Define  $E(x)$ to be true if Eq.~\eqref{E*condition} is satisfied by $x$;
$F(x)$ to be true if $x_4-x_3$ is even;
and $K(x)$ to be true if the total number of tokens $N(x)$ has residue 1 modulo 4,
i.e.
\[
N(x) = x_1 + x_2 + x_3 + x_4 + x_5 +x_6\equiv 1\pmod 4.
\]
We define the functions
\[
\begin{aligned}
 r(x)&=x_1+x_4-x_5, \\ s(x)  &= x_1 -x_2+x_3-x_4-x_5+x_6 , \\
u(x)&=s(x)-2r(x)+1.
\end{aligned}
\]
Note that since $N(x)$ is odd, $s(x)$ is also odd, and $u(x)$ is even.
Using these functions we define $T(x)=\min(s(x),u(x))$ and identify
three regions as follows:
\[
\begin{aligned}
  A = \{x: T(x)>0\},\quad
  B = \{x: T(x)=0\},\quad
  C = \{x: T(x)<0\}.
\end{aligned}
\]

\begin{theorem}
\label{t6odd}
  The P-positions of   {\sc Nim}$^1_{6, =2}[1]$
coincide with the set
\[
\cP =\{x: \big(E(x)\land (x\in A)\big) \lor \big(K(x)\land F(x)\land (x\in B)\big)
\lor\big( K(x)\land (x\in C)\big)\}.
\]
\end{theorem}

To get a description of P-positions for {\sc Nim}$^1_{5, =2}$ one
should set $x_1=0$ in Theorems~\ref{t6} and~\ref{t6odd} and shift
indices by 1. This gives two corollaries.

\begin{corollary}
\label{t5even}
The P-positions of   {\sc Nim}$^1_{5, =2}[0]$
are characterized by the parity vectors:
\begin{equation}
\label{eq-t5}
(e, o, o, o, o),  (e, e, e, e, e).
\end{equation}
\end{corollary}
\begin{proof}
  The two other parity vectors correspond to odd values of $x_1$ in
  Eq.~\eqref{eq-t6}.
\end{proof}

If $y=(y_1,y_2,y_3,y_4,y_5)$ and $x=(0,y_1,y_2,y_3,y_4,y_5)$, then define
\begin{align*}
 r'(y)&=r(x)=y_3-y_4, \\
 s'(y)&=s(x)= -y_1+y_2-y_3-y_4+y_5 , \\
u'(y)&=u(x)=s'(y)-2r'(y)+1, \\
E'(y)&=E(x)=(y_1\equiv y_3-y_2\equiv y_4-y_3 \equiv 0\pmod2, \\
K'(y)&=K(x)=y_1 + y_2 + y_3 + y_4 + y_5 \equiv 1\pmod 4.
\end{align*}
Since the components of $x$ and $y$ are in sorted order,
$r'(y)\leq 0$ and $u'(y)>s'(y)$, and so $T'(y)=\min(s'(y),u'(y))=s'(y)$.

\begin{corollary}
\label{t5}
The P-positions of   {\sc Nim}$^1_{5, =2}[1]$ are characterized by the formula
\[
(T'<0)K' \lor (T'>0)E'.
\]
\end{corollary}
\begin{proof}
The conditions of this corollary are obtained from the conditions of
Theorem~\ref{t6odd} by setting $x=(0,y_1,y_2,y_3,y_4,y_5)$ as in the definitions above.
Since $T'(y)=s'(y)\equiv y_1+y_2+y_3+y_4+y_5\equiv 1\pmod2$, the case $T'(y)=0$,
corresponding to $x\in B$, is impossible. The correspondence $x\in A$ if $T'(y)>0$ and
$x\in C$ if $T'(y)<0$ gives the P-positions for  {\sc Nim}$^1_{5, =2}[1]$.
\end{proof}

There are fewer invariant shifts in  {\sc Nim}$^1_{6, =2}[1]$
than in  {\sc Nim}$^1_{6, =2}[0]$.

\begin{proposition}
\label{c1}
Shifts by vectors
\[
\begin{aligned}
&y^{(1)}=(0,0,1,1,1,1),&& y^{(2)} =(1,1,1,1,2,2), \\
&y^{(3)} = (0,0,0,2,2,4),&& y^{(4)} =(0,2,2,2,2,4)
\end{aligned}
\]
preserve the sets described in Theorems~\textup{\ref{t6}} and~\textup{\ref{t6odd}}.
\end{proposition}

Thus, Theorems~\ref{t6} and~\ref{t6odd} imply that these vectors are
 invariant shifts in  {\sc Nim}$^1_{6, =2}$.

\begin{proof}
As mentioned above, shifts by these vectors preserve the set
described in Theorem~\ref{t6}. It remains to prove that they preserve
the set described in Theorem~\ref{t6odd}.

A shift by a vector $y^{(i)}$, $1\leq i\leq4$, preserves the parities of
$x_2-x_1$, $x_4-x_3$, $x_5-x_4-x_1$. So the conditions $E$ and $F$ are
invariant
under the shifts.

A shift by $y^{(i)}$, $1\leq i\leq4$, changes the total number of
tokens by a multiple of 4. So the condition $K$ is also invariant
under the shifts.

Note that
\[
\begin{aligned}
y^{(i)}_1 + y^{(i)}_4 - y^{(i)}_5&=0,\\
y^{(i)}_1 -y^{(i)}_2 +y^{(i)}_3 -y^{(i)}_4 -y^{(i)}_5 +y^{(i)}_6&=0.
\end{aligned}
\]
for  $1\leq i\leq4$. Thus the shifts preserve the values of each of the functions
$r$, $s$, $u$ and $T$.
\end{proof}

\medskip

The rest of the paper is organized as  follows.
In Section \ref{s1} we study P-shifts
(as well as some ``stronger" concepts).
In Section~\ref{ss21}  we prove Theorem~\ref{t6}, and  in Section~\ref{ss22}.
we prove  Theorem~\ref{t5}.

\section{Invariant shifts}
\label{s1}
Note that vector  $(0,0,0,2,2)$  is a P-shift
in {\sc Nim}$^1_{5, =2}$ assuming Corollaries~\ref{t6}
and~\ref{t6odd}.
The shift by this vector preserves the set described by
Eq.~\eqref{eq-t5}. It changes the total number of tokens by a multiple
of 4 and it preserves the condition $E'$ and the function $T'$. Thus it preserves
the set described in Corollaries~\ref{t6} and~\ref{t6odd}.

There are more examples:
$(1,1,1,1)$  and  $(1,1,1,1,1,1)$  are  P-shifts for
{\sc Nim}$^1_{4, =2}$  and  {\sc Nim}$^1_{6, =3}$, respectively;
$(0,1,1,1,1)$  and  $(0,1,1,1,1,1,1)$  are  P-shifts for
{\sc Nim}$^1_{5, =2}$  and  {\sc Nim}$^1_{7, =3}$, respectively;
$(0,2,2,2)$  and  $(0,2,2,2,2)$  are  P-shifts for
{\sc Nim}$^1_{4, =3}$  and  {\sc Nim}$^1_{5, =4}$, respectively.

Our next statement, generalizing all six examples,
will require the following definitions and notation.

Given a game  {\sc Nim}$^1_{n, =k}$  or {\sc Nim}$^1_{n, =k}[j]$,
a nonnegative nondecreasing  $n$-vector  $y$   will be called a
{\em  $g$-shift, or  $g^-$-shift, or $g^\pm$-shift}
if adding  $y$  to  any  position  $x$  of the considered game
preserves the Sprague-Grundy (SG) value,
mis\`{e}re SG value, or both values, respectively;
in other words, if
for any nonnegative nondecreasing $n$-vector $x$  we have
$g(x) = g(x+y)$, or $g^-(x)= g^-(x + y)$, or both equalities hold,
respectively.

The reader may recall
the definitions of the normal and mis\`{e}re SG functions
from the introduction of \cite{GHHC20}.
As usual, by $a^\ell$ we denote a symbol
(in particular, a number)  $a$  repeated  $\ell$  times.

\begin{theorem}
\label{p1}
Vectors  $(1^{2k})$, $(0, 1^{2k})$, and $(0, 2^k))$   are
$g^{\pm}$-shifts for the exact $k$-{\sc Nim} games:
{\sc Nim}$^1_{2k, =k}$, {\sc Nim}$^1_{2k+1, =k}$, and  {\sc Nim}$^1_{k+1, =k}$,
respectively.
\end{theorem}

Also, vector $(2^k)$  is a $g^{\pm}$-shift in {\sc Nim}$^1_{k, =k}$.
However, this game is trivial:
$g(x) = 1 - g^-(x) = x_1 \bmod 2$.

The {\em support}, $\supp(z)$,  of a nonnegative
$n$-vector $z$  is defined as the set of its positive coordinates.

Note also that in each of the three games considered in Theorem \ref{p1},
for the corresponding shift $y$  and any position  $x$  we  have:
\begin{itemize}
\item[(i)] the set-difference $\supp(x) \setminus \supp(y)$
contains at most one element;
\item[(ii)]  the sum of the coordinates of  $y$  equals  $2k$.  
\end{itemize}
Both observations will be essential in the proof.

\proof
We have to show that
$g(x) = g(x+y)$  and  $g^-(x) = g^-(x+y)$
for every position  $x$  in each of the three games considered.
We will prove all three claims simultaneously
by induction on the height  $h(x)$  
of a position  $x$.
Recall that  $h(x)$  is defined as the maximum number
of successive moves that can be made from  $x$.

To a move  $x \to x'$  we assign the move
from  $x + y \to  x''$  that reduces the same coordinates.
Then, $g(x') = g(x'')$  holds by the induction hypothesis,
implying that  $g(x+y) \geq g(x)$.
Assume for contradiction that  $g(x+y) > g(x)$.
Then there is a move  $x+y \to z'$  such that $g(z') = g(x)$.
If there  exists a  move  $x \to z''$
reducing the same coordinates then the previous arguments work
and by induction we obtain that  $g(z') = g(z'')$.
Hence, $g(z') \neq g(x)$, because otherwise
$g(x) = g(z'')$, while $x \to z''$  is a move,
which is a contradiction.

However, the required move  $x \to z''$  may fail to exist.
Obviously, in this case  $x_1 = 0$  must hold in all three cases.
For {\sc Nim}$^1_{2k+1, =k}$ and {\sc Nim}$^1_{k+1, =k}$
we also have  $y_1 = 0$  and hence,
move  $x+y \to z'$  cannot reduce  $x_1 + y_1 = 0$.
But then, in all three cases there exists a move  $z' \to x$.
Indeed, in {\sc Nim}$^1_{k+1, =k}$  we just repeat the previous move
reducing the same  piles.
In contrast, in {\sc Nim}$^1_{2k, =k}$  and  in {\sc Nim}$^1_{2k+1, =k}$
we make the ``complementary" move, reducing the $k$  piles
that were not touched by the previous move.
Note that in the second case we also do not touch the first pile,
because it is empty.
(Note also that both of the above conditions (i,ii)
are essential in all three cases.)

Finally, the existence of the move  $z' \to x$  implies that
$g(x) \neq g(z')$,
Thus, no move from  $x+y$  can reach the SG value  $g(x)$,
implying that equality  $g(x) = g(x+y)$  still holds.

The same arguments work in the mis\`{e}re case as well, since
the mis\`{e}re SG function  $g^-$  is defined by the same recursion as $g$,
differing from it only by the initialization.

\smallskip

It remains to verify the base of induction for both
$g$  and  $g^-$ and for all three games considered.
Note that all terminal positions
(from which there are no moves),
not only  $(0, \ldots, 0)$,  must be considered.
It is easily seen that the terminal positions in
{\sc Nim}$^1_{2k, =k}$, {\sc Nim}$^1_{2k+1, =k}$, and  {\sc Nim}$^1_{k+1, =k}$
are exactly the positions with at least $k+1, k+2,$ and $2$
zero coordinates, respectively.
For each game we have to verify that
$g(x+y) = g(x) = 0$   and
$g^-(x+y) = g^-(x) = 1$
for the corresponding shift $y$ and
any terminal position  $x$  of the game.
It is enough  to check that
\begin{itemize}
\item{} there is no move  $x+y \to x$;
\item{} for any move $x+y \to z'$  there exists
a move $z' \to x$  to a  terminal position  $x$;
\item{} there exists a move $x+y \to z''$  such that
each move from  $z''$  results in a terminal position.
\end{itemize}
We leave this tedious but simple case analysis to the careful reader.
\qed

\medskip

Note that Theorem \ref{p1} may simplify
the analysis of the three games considered.
Without loss of generality, we can restrict ourselves to the  positions $x$  such that

\begin{itemize}
\item{} $x_1 = 0$  in  {\sc Nim}$^1_{2k, =k}$;
in other words, we reduce
{\sc Nim}$^1_{2k, =k}$  to  {\sc Nim}$^1_{2k-1, =k}$;
(for  $k = 2$  see Theorem 4 
and Corollary 1  
in \cite{GHHC20});
\item{} $x_1 = x_2$  in  {\sc Nim}$^1_{2k+1, =k}$;
\item{} $x_2 \leq x_1 + 1$  for   {\sc Nim}$^1_{k+1, =k}$.
\end{itemize}

Let us mention a few  more observations related to P-shifts.

\smallskip

Vector $(0,0,0,2,2)$  is not a P-shift
in the mis\`{e}re version of  {\sc Nim}$^1_{5, =2}$.
For example, direct computations show that
$(3,3,3,4,8)$  is a P-position, while
$(3,3,3,6,10)$  is not.
We have no explicit formula for the P-positions
in this case.

\medskip

Vector $(0,0,0,0,4)$  is not a P-shift in  {\sc Nim}$^1_{5, =2}$.
For example,
$x' = (2,2,3,4,6)$  is a  P-position, while
$x'' = (2,2,3,4,10)$  is not, by Corollary~\ref{t5}.
Indeed,
$T'(x')<0$ ($-2+2-3-4+6<0$) and $K'(x')$ is true ($2+2+3+4+6 =
17\equiv1\pmod4$). On the other hand, $T(x'')>0$ is false
($-2+2-3-4+10>0$), and $E(x''_4-x''_3)$ is false.
Thus, in accordance with Corollary~\ref{t5},
$x'$ is a P-position, while  $x''$  is not.

\medskip

Finally, let us consider shift $(0,0,0,2,2,2)$  in {\sc Nim}$^1_{6, =3}$.
Our computations suggest the conjecture  that it is a P-shift in the subgames
{\sc Nim}$^1_{6, =3}[0]$  and  {\sc Nim}$^1_{6, =3}[1]$.
Yet, it is not a $g$-shift in these games, for example,

\smallskip

$3 = g(1,2,2,2,4,4) \neq g(1,2,2,4,6,6) = 5$;

$1 = g(1,2,3,3,3,4) \neq g(1,2,3,5,5,6) = 3$. \\
Also, it is not a P-shift in {\sc Nim}$^1_{6, =3}[2]$; for example,

$0 = g(0,7,7,7,7,10) \neq  g(0,7,7,9,9,12) = 3$. \\
Finally, it is not a $g^-$-shift,
and not even a  P-shift in the mis\'{e}re version of {\sc Nim}$^1_{6, =3}[j]$
for all  $j =0,1,2$.  For example,

\smallskip

$0 = g^-(1,2,3,3,3,3) \neq  g^-(1,2,3,5,5,5) = 1$;

$0 = g^-(1,2,3,3,3,4) \neq  g^-(1,2,3,5,5,6) = 3$;

$0 = g^-(0,1,2,2,2,4) \neq  g^-(0,1,2,4,4,6) = 3$.

\section{Proofs of Main Theorems} 
\label{ss2}

\subsection{Proof of Theorem \ref{t6}}
\label{ss21}
We have to prove that the set of P-positions
of {\sc Nim}$^1_{6, =2}[0]$  coincides with
the set $\cP$  satisfying (\ref{eq-t6}).
It is enough to show that
(I)  there is no move  $x \to x'$  such that
$x, x' \in \cP$  and
(II) for any  $x \not\in \cP$  there is a move
$x \to  x'$   such that  $x' \in \cP$.

\medskip

(I). By (\ref{eq-t6}), the Hamming distance between any two parity
vectors of distinct positions of $\cP$ is exactly $4$, while for any
move $x \to x'$ the Hamming distance between $p(x)$ and $p(x')$ is
exactly $2$, in accordance with the rules of the game.

\medskip

(II). Fix a position  $x = (x_1,\ldots,x_6) \not\in \cP$
in {\sc Nim}$^1_{6, =2}[0]$.
Since  $\sum_{i=1}^6 x_i$  is even, the number
$j$  of odd coordinates of  $x$  is even too and
we have to consider the following three cases.

\smallskip

$j = 6$, that is, all six piles are odd,
$p(x) = (o,o,o,o,o,o)$.
In this case we can reduce  $x_1$ and $x_2$,
getting a  move $x \to x'$  such that
$p(x') =  (e,e,o,o,o,o)$.

\smallskip

$j = 2$, that is, there are exactly two odd piles.
Reducing them we get a move  $x \to x'$  such  that
$p(x') =  (e,e,e,e,e,e)$.

\smallskip

$j = 4$, that is, there are four odd piles and two even.
Let us match coordinates  1 and 2, 3 and 4, 5 and 6.
Make the (unique) move  $x \to x'$  reducing in $x$
the smaller even pile and the odd one
that is matched with the larger even pile.
If the two even piles are of the same size,
$x_i = x_{i+1}$,  then we agree that
pile $i$  is smaller than pile  $i+1$.    

\smallskip

It is easy to  verify that in all the above cases
the chosen move  $x \to  x'$  is possible,
because it reduces (by one token) two non-empty piles;
furthermore, in all cases  $x' \in \cP$, by~(\ref{eq-t6}).
\qed

\subsection{Proof of Theorem \ref{t6odd}}
\label{ss22}

It is convenient to change coordinates for positions. We adopt
`differential coordinates':
\[
\begin{aligned}
&q_1= x_1,& &q_2 = x_2-x_1,& &q_3 = x_3 - x_2,\\
&q_4 = x_4-x_3,&  &q_5 = x_5 - x_4, & &q_6 = x_6-x_5.
\end{aligned}
\]
Recall that $x$ is assumed to be nondecreasing. Thus $q_i\geq0$ for all $i$.

We define a move, reducing two piles by one token each, to be \emph{legal} if the resulting
piles are still in nondecreasing order of size. For a move ${i,j}$, where $i<j$ this
condition is that $q_i>0$, and if $i<j-1$ then also $q_j>0$.

In these coordinates the  conditions used in Theorem~\ref{t5} are
expressed   as
follows:  $r(q)=q_1-q_5$; $s(q)= -q_2-q_4+q_6$;
$u(q)=-2q_1-q_2-q_4+2q_5+q_6+1$; and $T(q)=\min(s(q),u(q))$.
The three polyhedral  regions are as before:
\begin{equation}\label{62odd-regions}
\begin{aligned}
  &A = \{q: T(q)>0\},\\
  &B = \{q: T(q)=0\},\\
  &C = \{q: T(q)<0\}.
\end{aligned}
\end{equation}
The set described in the theorem is expressed in the differential
coordinates as
\begin{equation}\label{62odd-condition}
\cP =\{q: \big(E(q)\land (q\in A)\big) \lor \big(K(q)\land F(q)\land (q\in B)\big)
\lor\big( K(q)\land (q\in C)\big)\},
\end{equation}
where
\[
\begin{aligned}
E(q) &:= (q_2\equiv q_4\equiv q_1-q_5 \equiv 0 \pmod2),\\
F(q) &:= (q_4\equiv 0 \pmod 2) ,\\
K(q) &:= (N(q)\equiv -2q_1+q_2-q_4+2q_5+q_6\equiv 1\pmod4).
\end{aligned}
\]

We have to prove that the set of P-positions
of {\sc Nim}$^1_{6, =2}[1]$  coincides with
the set $\cP$.
The parity condition
\begin{equation}\label{parity-condition}
x_1+x_2+x_3+x_4+x_5+x_6 \equiv q_2+q_4+q_6\equiv 1\pmod2
\end{equation}
implies that $s(q)$ is odd and $u(q)$ is even.

We will use several simple observations on the
conditions~\eqref{62odd-condition} before and after a move.

\begin{proposition}\label{Kalternates}
Let $q\to q'$ be a move in  {\sc Nim}$^1_{6, =2}[1]$. Then
$K(q)\oplus K(q')=1$.
\end{proposition}
\begin{proof}
  The total number of tokens $N$ decreases by 2 after each move. So, the
  (odd) residue of $N$ modulo 4 alternates between 1 and 3.
\end{proof}

\begin{proposition}\label{parity-break}
  Let $q\to q'$ be a move in  {\sc Nim}$^1_{6, =2}[1]$. Then
$E(q)$ is true implies that  $ E(q')$ is false.
\end{proposition}
\begin{proof}
We repeat the argument from the proof of Theorem~\ref{t6}, case
(I). If the total number of tokens is odd and $E(q)$ is true, then the
parities of vector $x$ have the form $(a,a,b,b,c,\bar c)$ where $c=a\oplus b$. This implies
that the Hamming distance between any two parity
vectors of distinct positions of $\cP$ is exactly $4$, while for any
move $x \to x'$ the Hamming distance between $p(x)$ and $p(x')$ is
exactly $2$.
\end{proof}

In the case analysis below, Table~\ref{q-moves} is helpful. It
shows, for each move, the changes of $q$-coordinates, $s(q)$, and $u(q)$.

\begin{table}
\caption{Results of moves $q\to q'$ in `differential coordinates', $\Delta q_i = q'_i - q_i$,
$\Delta r = r(q')-r(q)$, $\Delta s = s(q')-s(q)$, $\Delta u = u(q')-u(q)$.}
\label{q-moves}
  \def\arraystretch{1.5}
\[
\begin{array}{|c|c|c|c|c|c|c||c||c||c||}
\hline
  \text{Moves}& \Delta q_1&\Delta q_2&\Delta q_3&
\Delta q_4&\Delta q_5&\Delta q_6&\Delta r&\Delta s&\Delta u \\
\hline
\{1,2\} & -1&0&+1&0&0&0&    -1&0&+2\\
\{1,3\} &-1&+1&-1&+1&0&0&  -1&-2&0\\
\{1,4\} & -1&+1&0&-1&+1&0& -2&0&+4\\
\{1,5\} &-1&+1&0&0&-1&+1&   0&0&0\\
\{1,6\} & -1&+1&0&0&0&-1&   -1&-2& 0\\
\{2,3\} &0 &-1&0&+1&0&0&     0&0&0\\
\{2,4\} &0&-1&+1&-1&+1&0& -1&+2&+4\\
\{2,5\} &0&-1&+1&0&-1&+1& +1&+2&0\\
\{2,6\} &0 &-1&+1&0&0&-1&   0&0&0\\
\{3,4\} &0 &0&-1&0&+1& 0&  -1&0&+2\\
\{3,5\}&0&0&-1&+1&-1&+1& +1&0&-2\\
\{3,6\} &0 &0&-1&+1&0&-1&   0&-2&-2\\
\{4,5\} &0 &0&0&-1&0& +1&   0&+2&+2\\
\{4,6\} &0 &0&0&-1&+1&-1& -1&0&+2\\
\{5,6\} &0 &0&0&0&-1&0&   +1&0&-2\\
\hline
\end{array}
\]
\end{table}

As in the proof of Theorem~\ref{t6}, it is enough to show that
(I)  there is no move  $q \to q'$  such that
$q, q' \in \cP$  and
(II) for any  $q \not\in \cP$  there is a move
$q \to  q'$   such that  $q' \in \cP$.

\medskip

We  prove (I) by contradiction. Suppose
that $q, q'\in \cP$ for a move $q\to q'$.
Propositions~\ref{Kalternates} and~\ref{parity-break} imply that
either (i) $K(q)\&\neg E(q)\&\neg K(q')\&E(q')$ or (ii) $\neg K(q)\&E(q)\& K(q')\&\neg E(q')$.

\medskip

\textsc{Case} (i): $K(q)\&\neg E(q)\&\neg K(q')\&E(q')$. \\
In this case, $T(q')>0$, $E(q')$, and either $(T(q)=0)\&F(q)$ or $T(q)< 0$.

If $(T(q)=0)\&F(q)$ then from $K(q)$ and $T(q)=u(q)=0$, we have
\begin{align*}
  &-2q_1+q_2-q_4+2q_5+ q_6   &\equiv&\ 1\pmod4, \mathrm{\ and}&\\
  &-2q_1-q_2-q_4+2q_5+ q_6+1 &=&\ 0,&
\end{align*}
which implies that $q_2$ is odd. Since $F(q)$ (i.e., $q_4$ is even) and $E(q')$,
any move taking $q$ to $q'$ must have $\Delta q_4$ even and $\Delta q_2$ odd.
From Table~\ref{q-moves} we see that there are just four moves with this property but
all of these have $\Delta u =0$, which contradicts the required increase in $T$.

Otherwise, $T(q)< 0$, so $u(q)\leq -2$ or $s(q)\leq -1$ but $u(q')\geq 2$ and $s(q')\geq 1$.
Since $\Delta u\leq 4$ and $\Delta s\leq 2$ for all the moves, either $u(q')=2$ or $s(q')=1$.
From $\neg K(q')\&E(q')$, we have
\begin{align*}
 N(q')\equiv -2q'_1+q'_2-q'_4+2q'_5+q'_6&\equiv3&\pmod4& \mathrm{\ and}\\
  q'_2\equiv q'_4\equiv q'_5-q'_1&\equiv 0& \pmod 2&.
\end{align*}
Note that
\begin{align*}
3\equiv N(q') &\equiv u(q')+2q'_2-1 &\pmod 4& \mathrm{\ and}\\
3\equiv N(q') &\equiv s(q')+2q'_2-2r(q')-2 &\pmod 4&,
\end{align*}
so $u(q')=2$ implies $q'_2$ odd, and $s(q')=1$ implies $q'_2-r(q')$ odd.
Either of these contradict $E(q')$ which has $q'_2$ and $r(q')$ both even.

\medskip

\textsc{Case} (ii): $\neg K(q)\&E(q)\& K(q')\&\neg E(q')$. \\
In this case, $T(q)>0$, $E(q)$, and either $(T(q')=0)\&F(q')$ or $T(q')< 0$.
The analysis is almost the
same as above, exchanging $q'$ for $q$, replacing ```increase'' by ``decrease'', and similar.

\bigskip

Now we prove (II). Thus, $q\notin\cP$ and we are going to indicate
a~move $q\to q'$ such that $q'\in\cP$.

\medskip

\textsc{Case $q\in C$.} This implies $\neg K(q)$.
By Proposition~\ref{Kalternates}, $q'\in\cP$ for any move $q\to q'$ preserving
region~$C$.

Table~\ref{q-moves} shows many moves that preserve region~$C$,
including$\{2,3\}$, $\{5,6\}$, and $\{1,6\}$. If $q_2>0$ or $q_5>0$ then
$\{2,3\}$ or $\{5,6\}$ is legal and we are done. Now we assume $q_2=q_5=0$.

If $q_6>0$ then $\{1,6\}$ is legal provided that $q_1>0$. If $q_1=0$ then
$u(q)= -q_4+q_6+1>s(q)= -q_4+q_6 = T(q) < 0$. Move $\{4,6\}$ is legal
since $q_4>q_6>0$ and gives $\Delta s=0$ from Table~\ref{q-moves}.
Therefore $s(q')=s(q)<0$ and $q'\in \cP$.

Otherwise, if $q_6=0$ then $s(q)=-q_4<0$ and $u(q)= -2q_1 -q_4 +1<0$.
If $q_1>0$ then $\{1,4\}$ is legal and gives $\Delta s=0$, thus preserving
region~$C$. However, if $q_1=0$ then $u(q)=-q_4+1<0$. Since $u$ is
always even, this implies $q_4\geq 3$. Move $\{4,5\}$ is legal and gives
$\Delta s=2$, so $s(q')= s(q)+2= -q_4+2<0$ and region~$C$ is preserved.

\medskip

\textsc{Case $q\in B$.}

\noindent In this case $T(q)=0$ and, since this is even,
$s(q)>u(q)=T(q)=0$. Therefore $r(q)>0$, i.e., $q_1>q_5\geq 0$,
and $q_6\geq s(q)>0$.

\medskip

\textsc{Subcase $q\in B$ and $\neg K(q)$.}

\noindent Proposition~\ref{Kalternates} proves $K(q')$.
Since $2q_2\equiv 1+N(q)\equiv 0 \pmod4$, $q_2$ is even.
We will show
that at least one of the moves $\{1,6\}$, $\{4,5\}$ or $\{4,6\}$ gives
$q'\in\cP$.

(i) If $q_4$ is even, we use $\{1,6\}$, which is legal since $q_1>0$ and $q_6>0$.
Table~\ref{q-moves} shows that $\Delta q_4=0$ and $\Delta T\leq 0$,
therefore $K(q')\&F(q')\&(q'\in B\cup C)$ and so $q'\in\cP$.

(ii) If $q_4$ is odd, then $q_4>0$ and both $\{4,5\}$ and $\{4,6\}$ are legal.
If $r(q)$ is even, then we see from Table~\ref{q-moves} that $\{4,5\}$ gives
$\Delta q_2=0$, $\Delta q_4= -1$, $\Delta r=0$, $\Delta T=2$,
so $E(q')\&(T(q')=2)$ and therefore $q'\in\cP$.
Otherwise, if $r(q)$ is odd then $\{4,6\}$ ensures in a similar way that
 $E(q')\&(T(q')>0)$ and therefore $q'\in\cP$.

\medskip

\textsc{Subcase $q\in B$ and $K(q)$.}

\noindent Since $2q_2\equiv 1+N(q)\equiv 2 \pmod4$, $q_2$ is odd,
and since $q\notin\cP$, $F$ is false, i.e., $q_4$ is odd.
Also $s(q)=-q_2-q_4+q_6$ is always odd, so $q_6$ is odd too.
We have $q_1>0$ and both of $q_2$ and $q_4$ odd,
so moves $\{1,4\}$ and $\{2,4\}$ are legal. Since $q\in B$, $s(q)>u(q)=0$ and
so Table~\ref{q-moves} shows that each of these moves
gives $T(q')=\min(s(q'),u(q'))>0$, i.e., $q'\in A$. To prove $q'\in\cP$ we have only
to show $E(q')$.

If $r(q)$ is even then $\{1,4\}$ gives
$\Delta q_2=+1$, $\Delta q_4= -1$, $\Delta r=-2$, so $E(q')$.

If $r(q)$ is odd then $\{2,4\}$ gives
$\Delta q_2=-1$, $\Delta q_4= -1$, $\Delta r=-1$, so $E(q')$.

\bigskip

\textsc{Case $q\in A$.}

\noindent In this case $E(q)$ is false. For the positions in this case,
we are going to indicate a move $q\to q'$ such that $q'\in A$ and
$E(q')$ is true,  so $q'\in\cP$.  Depending on the parities of $q_2$,
$q_4$, $r(q)$, suitable moves are listed in
Table~\ref{N-Pmoves}. It is easy to check using Table~\ref{q-moves}
that the listed moves  make  $q_2$, $q_4$ and $r(q)$ even, ensuring $E(q')$,
and furthermore they also satisfy $\Delta T\geq 0$, preserving $A$.
We now need to check that these moves are legal.

\medskip

\begin{table}
  \caption{Winning moves in the case $q\in A$}\label{N-Pmoves}
  \def\arraystretch{1.5}
\[
\begin{array}{|c|c|c|c|c|}
\hline
\text{Subcase}& q_2\bmod2  &q_4\bmod2 &r(q)\bmod2  &\text{moves to } \cP\\
\hline
A_1&0 & 0& 1& \{1,2\}\\ 
A_2&0 & 1& 0& \{4,5\}\\
A_3&0 & 1& 1& \{4,6\}\\
A_4&1 & 0& 0& \{2,6\}\\
A_5&1 & 0& 1& \{2,5\}\\
A_6&1 & 1& 0& \{2,3\}\\
A_7&1 & 1& 1& \{2,4\}\\
\hline
\end{array}
\]
\end{table}

In the  subcase $A_2$ the move $\{4,5\}$ is always legal, since $q_4>0$.
Similarly, for the move $\{2,3\}$ in the  subcase $A_6$
and  for the move $\{2,4\}$ in the  subcase $A_7$.
Other subcases require a more detailed  analysis.

\medskip

\textsc{Subcase $A_1$.} If $q_1>0$ then $\{1,2\}$ is legal,
otherwise $q_1=0$. This implies that $q_5>0$ since
$r(q)=q_1-q_5$ is odd in this subcase. Therefore move $\{5,6\}$ is
legal and changes the parities to $q'_2 =q_2 \equiv q'_4=q_4\equiv
(q'_5-q'_1) = q_5-1-q_1\equiv 0\pmod2 $. Thus $E(q')$ is true,
but we still need to prove that  $q'\in A$. In general, this move does not
preserve the region $A$: $\Delta s=0$ but $\Delta u=-2$.  Since $r(q)=q_1-q_5\leq -1$ and $T(q)=s(q)\geq 1$,
we have $u(q) = s(q)-2r(q)+1\geq 4$, and so $T(q')>0$ and $q'\in A$.

\medskip

\textsc{Subcase $A_3$.} Since $s(q) \geq 0$ and $q_4$ is odd,
$q_6>0$ in this case. Thus the move $\{4,6\}$ is legal.

\medskip

\textsc{Subcase $A_4$.} Again, since $s(q) \geq0$ and $q_2$ is odd, $q_6>0$
in this case, and the move $\{2,6\}$ is legal.

\medskip

\textsc{Subcase $A_5$.}  If $\{2,5\}$ is legal then we are done.
Otherwise, $q_5=0$, since $q_2$ is odd in this case.
This implies that  $q_1>0$ since $r(q)=q_1-q_5$ is odd.
Again, $q_6>0$ because $q_2$ is odd and $s(q)\geq0$. Thus, the move
$\{1,6\}$ is legal and it changes the parities to $q'_2 =q_2+1 \equiv q'_4=q_4\equiv
r(q') = r(q)-1\equiv 0\pmod2 $. Thus $E(q')$ is true. For $\{1,6\}$,
$\Delta s=-2$ and $\Delta u=0$.
From $u(q)=s(q)-2r(q)+1\geq 2$ and $r(q)>0$ we conclude that $s(q)\geq3$
and so $s(q')\geq 1$. Therefore $T(q')>0$ and $q'\in A$.
\qed

\section{Open questions and conjectures}

P-positions of {\sc Nim}$^1_{6,=2}$ are characterized
by the explicit formulas  of Theorems~\ref{t6} and~\ref{t5}.
Although they look complicated, they immediately provide
a polynomial time algorithm to solve the following problem:
given a nonnegative integer vector  $(x_1,x_2,x_3,x_4,x_5,x_6)$,
whose coordinates $x_i$  are represented in binary, check whether
 $(x_1,x_2,x_3,x_4,x_5,x_6)$ is a P-position in {\sc Nim}$^1_{6, =2}$.

Note that the set of P-positions of {\sc Nim}$^1_{6, =2}$ is semilinear.
Recall that a~\emph{semilinear set} is
a set of vectors from $\NN^d$ that can be expressed in Presburger
arithmetic; see, e.g.,~\cite{Haase}.  Presburger arithmetic admits
quantifier elimination.  So, a semilinear set can be expressed as a
finite union of solutions of systems of linear inequalities and
equations modulo some integer (they are fixed for the set).

This observation leads to a natural  conjecture.

\begin{conjecture}\label{weakest}
For any $n$, $k$  P-positions of {\sc Nim}$^1_{n,=k}$ form a
semilinear set.
\end{conjecture}

Conjecture~\ref{weakest} implies that there exists a polynomial time
algorithm deciding  P-positions of {\sc Nim}$^1_{n,=k}$ for all $n$,
$k$.

It is well-known that evaluating formulas in Presburger arithmetic is a
very hard problem. The first doubly exponential time bounds for this
problem were established by Fischer and Rabin~\cite{FR74}.
Berman proved that the problem requires doubly exponential space~\cite{B80}.

But these hardness results are irrelevant for our needs. In solving
{\sc Nim}$^1_{n,=k}$
we deal with sets of dimension $O(1)$ and may hardwire the
description of the set of P-positions in an algorithm solving the
game. Verifying a linear inequality takes polynomial time as well as
verifying an equation modulo an integer. Thus, any semilinear set of
dimension $O(1)$ can be recognized by a polynomial time algorithm.

Note that for a wider class of games, so-called multidimensional
subtraction games with a~fixed difference set, P-positions can be not
semilinear. Moreover, it was recently proven that
there is no  polynomial time algorithm  solving
a specific game of this sort~\cite{GV20}.

In the game from~\cite{GV20} it is allowed to add tokens to some piles
by a move, provided that the total number of tokens is strictly
decreasing.
We consider this feature crucial for  hardness results.

To put it more formally,  we need to make formal  definitions.
A game from the class FDG is specified by a finite set
$\D\subset\ZZ^{d}$ which is called  the \emph{difference set}. We require that
\begin{equation*}
\sum_{i=1}^d a_i >0
\end{equation*}
for each $(a_1,\dots,a_d)\in\D$.

Positions of the game are vectors
$x=(x_1,\dots,x_d)$ with nonnegative coordinates.
A~move from $x$ to $y$ is possible if  $x-y\in \D$.

It is proven in~\cite{GV20} that for some constant $d$ there exists a
set $\D$ such that there is no algorithm solving the game with the difference
set $\D$ and running in time $O(2^{n/16})$, where $n$ is the input size
(the length of the binary representation of the position $(x_1,\dots, x_d)$).

Also, it was  proved by Larsson and W\"astlund~\cite{LW13} that the equivalence problem for FDG is undecidable.
The existence of difference vectors with negative
coordinates is essential for both results.

 {\sc Nim}$^1_{n, =k}$ belongs to the class FDG: difference
vectors are (0,1)-vectors with exactly $k$ coordinates equal to~1.
The solution of {\sc Nim}$^1_{6, =2}$ encourages us to suggest a
stronger conjecture.


\begin{conjecture}\label{strong}
For any FDG game such that $a_i\ge 0$ for each
$(a_1,\dots,a_d)\in\D$, the set of P-positions is semilinear.
\end{conjecture}

Conjecture~\ref{strong} also implies that the equivalence problem for
FDG with nonnegative difference vectors is decidable, since
Presburger arithmetic is decidable.

Why do we believe in Conjecture~\ref{strong} in spite of results
from~\cite{LW13} and~\cite{GV20}?
The first idea is to apply  inductive arguments  in the
case of nonnegative difference vectors.
Note that Conjecture~\ref{strong} holds for a 1-dimensional FDG.
If a coordinate becomes zero it remains zero during the rest of the
game (if all difference vectors are nonnegative). So to prove
Conjecture~\ref{strong} one needs to provide a reduction in dimension
of a game.

This idea does not work in a straightforward manner. It can be shown
that arbitrary semilinear boundary conditions do not imply semilinearity of the
solution for the recurrence determining the set of
P-positions~\cite{Raskin20}. For the boundary conditions determined by
FDG games the question is open and we still have a hope to develop
a dimension reducing technique.

In any case, characterization of FDG games with semilinear P-positions
would be an important problem for future research.

\section*{Acknowledgements}
The paper was prepared within the framework
of the HSE University Basic Research Program;
the  fifth author were funded partially by the RFBR grant 20-01-00645.

\end{document}